\documentclass[11pt,leqno]{article}
\usepackage{graphicx, amsfonts, amsthm, amsxtra, amssymb, verbatim}
\usepackage[mathscr]{euscript}

\begin{document}
\def\podu{{\sf pd}}   
\def\per{{\sf pm}}      
\def\perr{{\sf q}}        
\def\perdo{{\cal K}}   
\def\sfl{{\mathrm F}} 
\def\sp{{\mathbb S}}  
 
\newcommand\diff[1]{\frac{d #1}{dz}} 
\def\End{{\rm End}}              
\def\hol{{\rm Hol}}
\def\sing{{\rm Sing}}            
\def\spec{{\rm Spec}}            
\def\cha{{\rm char}}             
\def\Gal{{\rm Gal}}              
\def\jacob{{\rm jacob}}          
\def\tjurina{{\rm tjurina}}      
\newcommand\Pn[1]{\mathbb{P}^{#1}}   
\def\Ff{\mathbb{F}}                  
\def\Z{\mathbb{Z}}                   
\def\Gm{\mathbb{G}_m}                 
\def\Q{\mathbb{Q}}                   
\def\C{\mathbb{C}}                   
\def\O{{\cal O}}                     
\def\as{\mathbb{U}}                  
\def\ring{{\mathsf R}}                         
\def\R{\mathbb{R}}                   
\def\N{\mathbb{N}}                   
\def\A{\mathbb{A}}                   
\def\uhp{{\mathbb H}}                
\newcommand\ep[1]{e^{\frac{2\pi i}{#1}}}
\newcommand\HH[2]{H^{#2}(#1)}        
\def\Mat{{\rm Mat}}              
\newcommand{\mat}[4]{
     \begin{pmatrix}
            #1 & #2 \\
            #3 & #4
       \end{pmatrix}
    }                                
\newcommand{\matt}[2]{
     \begin{pmatrix}                 
            #1   \\
            #2
       \end{pmatrix}
    }
\def\ker{{\rm ker}}              
\def\cl{{\rm cl}}                
\def\dR{{\rm dR}}                

\def\hc{{\mathsf H}}                 
\def\Hb{{\cal H}}                    
\def\GL{{\rm GL}}                
\def\pese{{\sf P}}                  
\def\pedo{{\cal  P}}                  
\def\PP{\tilde{\cal P}}              
\def\cm {{\cal C}}                   
\def\K{{\mathbb K}}                  
\def\k{{\mathsf k}}                  
\def\F{{\cal F}}                     
\def\M{{\cal M}}
\def\RR{{\cal R}}
\newcommand\Hi[1]{\mathbb{P}^{#1}_\infty}
\def\pt{\mathbb{C}[t]}               
\def\W{{\cal W}}                     
\def\gr{{\rm Gr}}                
\def\Im{{\rm Im}}                
\def\Re{{\rm Re}}                
\def\depth{{\rm depth}}
\newcommand\SL[2]{{\rm SL}(#1, #2)}    
\newcommand\PSL[2]{{\rm PSL}(#1, #2)}  
\def\Resi{{\rm Resi}}              

\def\L{{\cal L}}                     
\def\Aut{{\rm Aut}}              
\def\any{R}                          
\newcommand\ovl[1]{\overline{#1}}    

\def\T{{\cal T }}                    
\def\tr{{\mathsf t}}                 
\newcommand\mf[2]{{M}^{#1}_{#2}}     
\newcommand\bn[2]{\binom{#1}{#2}}    
\def\ja{{\rm j}}                 
\def\Sc{\mathsf{S}}                  
\newcommand\es[1]{g_{#1}}            
\newcommand\V{{\mathsf V}}           
\newcommand\WW{{\mathsf W}}          
\newcommand\Ss{{\cal O}}             
\def\rank{{\rm rank}}                
\def\Dif{{\cal D}}                   
\def\gcd{{\rm gcd}}                  
\def\zedi{{\rm ZD}}                  
\def\BM{{\mathsf H}}                 
\def\plf{{\sf pl}}                             
\def\sgn{{\rm sgn}}                      
\def\diag{{\rm diag}}                   
\def\hodge{{\rm Hodge}}
\def\HF{{\sf F}}                                
\def\WF{{\sf W}}                               
\def\HV{{\sf HV}}                                
\def\pol{{\rm pole}}                               
\def\bafi{{\sf r}}
\def\codim{{\rm codim}}                               
\def\id{{\rm id}}                               
\def\gms{{\sf M}}                           
\def\Iso{{\rm Iso}}                           
\newtheorem{theo}{Theorem}
\newtheorem{exam}{Example}
\newtheorem{coro}{Corollary}
\newtheorem{defi}{Definition}
\newtheorem{prob}{Problem}
\newtheorem{lemm}{Lemma}
\newtheorem{prop}{Proposition}
\newtheorem{rem}{Remark}
\newtheorem{conj}{Conjecture}
\newtheorem{calc}{}

\begin{center}
{\LARGE\bf  Eisenstein  type series for Calabi-Yau varieties 
\footnote{ 
Math. classification: 14N35, 
14J15, 32G20
}
}
\\
\vspace{.25in} {\large {\sc Hossein Movasati}} \\
Instituto de Matem\'atica Pura e Aplicada, IMPA, \\
Estrada Dona Castorina, 110,\\
22460-320, Rio de Janeiro, RJ, Brazil, \\
{\tt www.impa.br/$\sim$ hossein, hossein@impa.br} 
\end{center}
\begin{abstract}
In this article we introduce an ordinary differential equation associated to the one parameter family of Calabi-Yau varieties which is mirror dual to the universal family of 
smooth quintic three folds. It is satisfied by seven functions written in the $q$-expansion form and the Yukawa coupling turns out to be rational in these functions. We prove 
that these  functions are algebraically independent over the field of complex numbers, and hence, the algebra generated by such functions can be interpreted as 
the theory of quasi-modular forms attached to the one parameter family of Calabi-Yau varieties. 
Our result is a reformulation and realization of a problem of Griffiths around seventies on the existence of automorphic functions for the moduli of polarized Hodge structures. 
It is a generalization of the Ramanujan differential equation satisfied by three Eisenstein series.   
\end{abstract}
\section{Introduction}
Modular and quasi modular forms as generating functions count very unexpected objects beyond the scope of 
analytic number theory. There are many examples for supporting this fact. The Shimura-Taniyama conjecture, 
now the modularity theorem, states that the generating function for 
counting $\mathbb F_p$-rational points of an elliptic curve over $\Z$ for different primes $p$, is essentially a modular form. Monstrous 
moonshine conjecture, now Borcherds theorem, relates the coefficients of the $j$-function with the representation dimensions 
of the monster group. Counting ramified coverings of an elliptic curve with a fixed ramification data leads us to quasi modular 
forms. 

In the context of Algebraic Geometry, the theory of modular forms is attached to elliptic curves and in a 
similar way the theory of Siegel and Hilbert modular forms is attached to polarized abelian varieties. A naive mind may dream of other 
modular form 
theories attached to other varieties of a fixed topological type. An attempt to formulate such theories was first done around 
seventies by P. Griffiths  in the framework of Hodge structures, see \cite{gr70}. However, such a formulation leads us to the notion of 
automorphic cohomology which has  lost the generating function role of modular forms. Extending the algebra of any type of 
modular forms into an algebra of quasi modular forms, which is closed under canonical derivations, seems to be indispensable for 
further generalizations.   

In 1991 there appeared the article of Candelas, de la Ossa, Green and Parker, in which they  calculated in the framework of mirror symmetry 
a generating function, called the Yukawa coupling,  which predicts the number of rational curves of a fixed degree in a generic quintic 
three fold. 
From mathematical point of view, 
the finiteness is still a conjecture carrying the name of Clemens. Since then there was some effort to express the Yukawa 
coupling in terms of classical modular or quasi modular forms, however, there was no success. The Yukawa coupling is 
calculated from the periods of a one parameter family of Calabi-Yau varieties and this suggests that there must be 
a theory of quasi  modular forms attached to this family.  The main aim of the present text is to realize the construction of 
such a theory.  
     
Consider the following ordinary differential equation in seven variables $t_0,t_1,\ldots,t_4,t_5,t_6$:
\begin{equation}
\label{lovely}
\left \{ \begin{array}{l}
\dot t_0=\frac{1}{t_5}(\frac{6}{5}t_0^5+\frac{1}{3125}t_0t_3-\frac{1}{5}t_4)
\\
\dot t_1=\frac{1}{t_5}
(-125t_0^6+t_0^4t_1+125t_0t_4+\frac{1}{3125}t_1t_3) 
\\
\dot t_2=\frac{1}{t_5}(-1875t_0^7-\frac{1}{5}t_0^5t_1+2t_0^4t_2+1875t_0^2t_4+\frac{1}{5}t_1t_4+\frac{2}{3125}t_2t_3) 
\\
\dot t_3=\frac{1}{t_5}
(-3125t_0^8-\frac{1}{5}t_0^5t_2+3t_0^4t_3+3125t_0^3t_4+\frac{1}{5}t_2t_4+
\frac{3}{3125}t_3^2)
\\
\dot t_4=\frac{1}{t_5}
(5t_0^4t_4+\frac{1}{625}t_3t_4)
\\
\dot t_5=\frac{t_6}{t_5}
\\
\dot t_6=(-\frac{72}{5}t_0^8-\frac{24}{3125}t_0^4t_3-\frac{3}{5}t_0^3t_4-\frac{2}{1953125}t_3^2)+
\frac{t_6}{t_5}(12t_0^4+\frac{2}{625}t_3)
\end{array} \right.,
\end{equation}
where 
$$
\dot t=5q\frac{\partial t}{\partial q}.
$$
We write each $t_i$ as a formal power series in $q$, $t_i=\sum_{n=0}^\infty t_{i,n}q^n$ and substitute in the above differential equation and we see that it determines all the coefficients $t_{i,n}$ uniquely with the initial values:
\begin{equation}
 \label{22july2010}
t_{0,0}=\frac{1}{5}, \ t_{0,1}=24,\  t_{4,0}=0
\end{equation}
and assuming that $t_{5,0}\not =0$. After substitution we get the two possibilities  $0, \frac{-1}{3125}$ for  $t_{5,0}$, and $t_{i,n},\ n\geq 2$  is given in terms of  $t_{j,m}, \ \ j=0,1,\ldots,6,\ \  m<n$. See \S \ref{conjecture} for the first eleven coeffiecients of $t_i$'s. 
We calculate  the expression $\frac{-(t_4-t_0^5)^2}{625t_5^3}$ and write it in  Lambert series form. It turns out that 
$$
\frac{-(t_4-t_0^5)^2}{625t_5^3}=5+2875 \frac{q}{1-q}+ 609250\cdot 2^2\frac{q^2}{1-q^2}+
\cdots+ n_d d^3\frac{q^d}{1-q^d}+\cdots.
$$
Let $W_{\psi}$ be the variety  obtained by the resolution of singularities of the following quotient:
\begin{equation}
 \label{shahva}
W_\psi:=\{[x_0:x_1:x_2:x_3:x_4]\in \mathbb P ^ 4\mid x_0^5+  x_1^5+ x_2^5+ x_3^5+ x_4^5-5\psi x_0x_1x_2x_3x_4=0\}/G, 
\end{equation}
where $G$ is the group 
$$
G:=\{(\zeta_1,\zeta_2,\cdots,\zeta_5)\mid  \zeta_i^5=1, \ \zeta_1\zeta_2\zeta_3\zeta_4\zeta_5=1 \}
$$  
acting in a canonical way. The family $W_\psi$ is Calabi-Yau and it is mirror dual to the universal family of quintic varieties in $\Pn 4$. 
\begin{theo}
\label{29.03.10}
The quantity $\frac{-(t_4-t_0^5)^2}{625t_5^3}$ is the Yukawa coupling associated to the 
family of Calabi-Yau varieties $W_\psi$.
\end{theo}
The $q$-expansion of the Yukawa coupling is calculated by Candelas, de la Ossa, Green, 
Parkers in \cite{can91}, see also \cite{mo92}. Using physical arguments they showed that $n_d$ must be the number of degree $d$ rational curves inside a generic quintic three fold. 
However, from mathematical point of view we have the Clemens conjecture which claims that there are finite number of such curves for all $d\in \mathbb N$. This conjecture is established for $d\leq 9$ and remains open for $d$ equal to $10$ or bigger than it. The Gromov-Witten  invariants 
$N_d$  can be calculated using the well-known formula $
N_d=\sum_{k\mid d}\frac{n_{d/k}}{k^3}$. 
The numbers $n_d$ are called instanton numbers or BPS states degeneracies.  
The $\C$-algebra generated by $t_i$'s can be considered as the theory of quasi modular forms attached to the family $W_\psi$.  We prove:
\begin{theo}
\label{29march10}
The functions $t_i, i=0,1,\ldots,6$ are algebraically independent over $\C$, this means that there is no polynomial $P$ in seven variables 
and with coefficients in $\C$ such that $P(t_0,t_1,\cdots,t_6)=0$. 
\end{theo}

Calculation of instanton numbers by our differential equation (\ref{lovely}) or by using periods, see \cite{can91,mo92}, 
or by constructing moduli spaces of maps from curves to projective spaces,  see \cite{kon95},  leads to the fact that they are rational 
numbers. It is conjectured that all $n_d$'s are integers (Gopakumar-Vafa conjecture). 
Some partial results regarding this conjecture is established recently by Kontsevich-Schwarz-Vologodsky and Krattenthaler-Rivoal.

All the quantities $t_i,\ i=0,1,\ldots,6$ and $q$ can be written in terms of the periods of the family $W_\psi$. The differential form
$$
\Omega=\frac{x_4dx_0\wedge dx_1\wedge dx_2}{\frac{\partial Q}{\partial x_3}}, 
$$ 
where $Q$ is the defining polynomial of $W_\psi$, induces a holomorphic 3-form in $W_\psi$ which we denote it by the same letter $\Omega$. 
Note that $5\psi\Omega$ is 
the standard choice of a holomorphic 
differential $3$-form on $W_\psi$ (see \cite{can91}, p. 29). Let also  $\delta_1,\delta_2,\delta_3,\delta_4$ be a particular basis of 
$H_3(W_\psi,\Q)$ 
which will be explained 
in \S \ref{5aug2010}, and 
$$
x_{ij}=\frac{\partial^{j-1}}{\partial \psi^{j-1}}\int_{\delta_i}\Omega,\ i,j=1,2,3,4.
$$ 
\begin{theo}
\label{main3}
 The $q$-expansion of $t_i$'s are convergent and if  we set $q=e^{2\pi i\frac{x_{21}}{x_{11}}}$ then
\begin{eqnarray*}
t_0 &=& a^{-3} 
{\psi}x_{11} \\
t_1 &=&a^{-6}
625x_{11}\left (5\psi^3x_{12}+5\psi^4x_{13}+(\psi^5-1)x_{14}\right )
\\
t_2 &=&a^{-9}
(-625)
x_{11}^2
\left ( 5\psi^3x_{11}+(\psi^5-1)x_{13}\right )
\\
t_3 &=&a^{-12}
625
x_{11}^3
\left (-5\psi^4x_{11}+(\psi^5-1)x_{12}\right )
\\
t_4 &=&a^{-15}
x_{11}^5\\
 t_5 &=&a^{-11}
\frac{-1}{5}(\psi^5-1)
x_{11}^2
\left (x_{12}x_{21}-x_{11}x_{22}\right )
\\
  t_6 &=& 
 a^{-23}\frac{1}{25}(\psi^5-1)
x_{11}^5
\left (
-5\psi^4x_{11}x_{12}x_{21}-2(\psi^5-1)x_{12}^2x_{21}-(\psi^5-1)x_{11}x_{13}x_{21}+ \right. 
\\ & & \left. 5\psi^4x_{11}^2x_{22}+2(\psi^5-1)x_{11}x_{12}x_{22}+(\psi^5-1)x_{11}^2x_{23}\right ), 
\end{eqnarray*}
where $a=\frac{2\pi i}{5}$.
\end{theo}
 Once all the above quantities are given, using the Picard-Fuchs of $x_{i1}$'s, see (\ref{31aug2010}), one can check easily that 
they satisfy the ordinary differential equation (\ref{lovely}). However, how we have calculated them, and in particular moduli interpretation 
of $t_i,\ i=0,1,\ldots,6$, will be explained throughout the present text.    

This work can be considered as a realization of a problem of Griffiths around 1970's on the automorphic form theory for the moduli of 
polarized Hodge structures, see \cite{gr70}. In our case $H^3(W_\psi,\C)$ is of dimension $4$ and it carries a Hodge decomposition with Hodge numbers $h^{30}=h^{21}=h^{12}=h^{03}=1$. As far as I know, this is the first case of automorphic function theory for families of varieties for which 
the corresponding Griffiths period domain is not Hermitian symmetric. It would be of interest to see how the results of this 
paper fit into the automorphic cohomology theory of Griffiths or vice versa.

Here, I would like to say some words about the methods used in the present text and whether one can apply them to other families of varieties. We construct affine coordinates  
for  the moduli of the variety $W_\psi$ enhanced with elements in its third de Rham cohomology, see \S \ref{moduli1}, \S\ref{moduli2} and \S\ref{moduli3}. 
Such a moduli turns out to be of dimension seven and  such coordinates, say  $t_i,\ i=0,1,\ldots,6$,  have certain automorphic properties with respect to the action
of an algebraic group (the action of discrete groups in the classical theory of automorphic functions is replaced with the action of algebraic groups). 
We use the Picard-Fuchs equation of the periods of $\Omega$ and calculate the Gauss-Manin connection  (see for instance \cite{kaod68}) of the universal family of Calabi-Yau varieties over 
the mentioned moduli space. The ordinary differential equation (\ref{lovely}), seen as a vector field on the moduli space, has some nice properties with respect to the Gauss-Manin connection which determines it 
uniquely. A differential equation of type (\ref{lovely}) can be introduced for other type of varieties, see \cite{ho06-1}, however, whether it has a particular solution with a reach enumerative geometry behind, depends 
strongly on some integral monodromy conditions, see \S\ref{12may2010}, \S \ref{polyrel} and \S\ref{5aug2010}. For the moment I suspect that the methods introduced in this article can be generalized
to arbitrary families of Calabi-Yau varieties and even  to some other cases where the geometry is absent, see for instance the list of Calabi-Yau operators in \cite{almzud, vanvan} and  a table of 
mirror consistent monodromy representations in \cite{dormor}. Since the theory of Siegel modular forms is well developed and in light of the recent work \cite{do10-1}, 
see also the references within there, the case of K3  surfaces is quit promising.  In the final steps of the present article Charles Doran informed me of the results obtained 
by Yamaguchi and Yau in \cite{yam04}.  
This and other connections with mathematical physics will be explored in forthcoming articles.   

We have calculated the differential equation (\ref{lovely}) and the first coefficients of $t_i$ by Singular, see \cite{GPS01}. 
The reader who does not want to calculate everything by his 
own effort can obtain the corresponding Singular code from my web page.



\section{Quasi modular forms}
\label{modularforms}
The differential equation (\ref{lovely}) is a generalization of the Ramanujan differential equation
\begin{equation}
\label{raman}
 \left \{ \begin{array}{l}
\dot t_1=t_1^2-\frac{1}{12}t_2 \\
\dot t_2=4t_1t_2-6t_3 \\
\dot t_3=6t_1t_3-\frac{1}{3}t_2^2
\end{array} \right. \ \ \ \ \  \dot t=12 q\frac{\partial }{\partial q}
\end{equation}
which is satisfied by the Eisenstein series:
\begin{equation}
\label{eisenstein}
t_i=a_k{\Big (}1+b_k\sum_{d=1}^\infty d^{2k-1}\frac{q^{d}}{1-q^d}{\Big )},\ \  k=1,2,3,
\end{equation}
where 
$$
(b_1,b_2,b_3)=
(-24, 240, -504),\ \ (a_1,a_2,a_3)=(1,12 ,8).
$$
We have calculated  (\ref{lovely}) using the Gauss-Manin connection of the family $W_\psi$ which is essentially the Picard-Fuchs differential equation of 
the holomorphic differential form of the family $W_\psi$. This is done in a similar way as we calculate (\ref{raman}) from the Gauss-Manin connection of a family of elliptic curves, 
see \cite{ho06-2,ho06-3}. The general theory of differential equations of type (\ref{lovely}) and (\ref{raman}) is developed in \cite{ho06-1}. Relations between the Gauss-Manin connection and 
Eisenstein series appear  in  the appendix of \cite{ka73}.  Let $g_1,g_2,g_3$ be the Eisenstein series (\ref{eisenstein}). The  $\C$-algebra 
$\C[g_1,g_2,g_3]$ 
is freely generated by $g_1,g_2,g_3$. With $\deg(g_i)=i, i=1,2,3$, its homogeneous pieces are quasi-modular forms over $\SL 2\Z$.  It can be shown that any other quasi-modular form for subgroups of $\SL 2\Z$ with finite index, is in the algebraic closure of $\C(g_1,g_2,g_3)$.

\section{Moduli space, I}
\label{moduli1}
In the affine coordinates $x_0=1$, the variety  $W_\psi$ 
is given by:
$$
\{(x_1,x_2,x_3,x_4)\in \mathbb C ^ 4\mid f=0\}/G,
$$
where
$$
f=-z-x_1^5-x_2^5-x_3^5-x_4^5+5x_1x_2x_3x_4
$$
and we have introduced a new parameter $z:=\psi^{-5}$. We also use  $W_{1,z}$  to denote the variety $W_\psi$. 
For $z=0,1,\infty$ the variety $W_{1,z}$ is singular and for all others it is a smooth variety of
 complex dimension $3$. From now on, by $W_{1,z}$ we mean a smooth one.
Up to constant there is a unique holomorphic three form on $W_{1,z}$ which is given by
$$
\eta=\frac{ dx_1\wedge dx_2\wedge dx_3\wedge dx_4}{df}.
$$
 Note that the pair $(W_{1,z},5\eta)$ is isomorphic to 
$(W_\psi,5\psi\Omega)$, with $\Omega$ as in the Introduction. The later is used in \cite{can91} p. 29. 
The third de Rham cohomology of $W_{1,z}$, namely $H^{3}_\dR(W_{1,z})$,  carries a Hodge decomposition with Hodge numbers $h^{30}=h^{21}=h^{12}=h^{03}=1$. 
By Serre duality $H^2(W_{1,z},\Omega^1)\cong H^1(W_{1,z},\Theta)$, where $\Omega^1$ (rep. $\Theta$) is
 the sheaf of holomorphic differential $1$-forms (resp. vector fields) on $W_{1,z}$. 
Since $h^{21}=\dim_\C H^2(W_{1,z},\Omega^1)=1$, the deformation space of $W_{1,z}$ is one dimensional. 
This means that $W_{1,z}$ can be deformed only through the parameter $z$. In fact $z$ is the classifying function of such varieties. Note that the finite values of $z$ does not cover the smooth variety $W_{\psi},\ \psi=0$.

Let us take the polynomial ring $\C[t_0,t_4]$ in two variables $t_0,t_4$ (the variables $t_1,t_2$ and $t_3$ will appear later). It can be seen easily that
the moduli $S$ of the pairs $(W,\omega)$, where $W$ is as above and $\omega$ is a 
holomorphic differential form on $W$, is isomorphic to
$$
S\cong \C^2\backslash\{(t_0^5-t_4)t_4=0\},
$$ 
where we send the pair $(W_{1,z},a\eta)$ to $(t_0,t_4):=(a^{-1}, za^{-5})$. The multiplicative group $G_m:=\C^*$ acts on $S$ by:
 $$
(W,\omega)\bullet k=(W,k^{-1}\omega),\ k\in G_m,\ (W,\omega)\in S.
$$
In coordinates $(t_0,t_4)$ this corresponds to
\begin{equation}
\label{poloar}
(t_0,t_4)\bullet k=(kt_0,k^{5}t_4),\ (t_0,t_4)\in S,\ k\in G_m.
\end{equation}
We denote by $(W_{t_0,t_{4}},\omega_1)$ the pair $(W_{1,\frac{t_4}{t_0^{5}}},t_0^{-1}\eta)$. The one parameter family $W_{1,z}$ (resp. $W_\psi$) can be recovered by putting  
$t_0=1$ and $t_4=z$ (resp. $t_0=\psi$ and $t_4=1$). In fact, the pair $(W_{t_0,t_4},\omega_1)$ in the affine chart $x_0=1$ is given by:
\begin{equation}
\left (\{ f_{t_0,t_4}(x)=0\}/G,\ \frac{ dx_1\wedge dx_2\wedge dx_3\wedge dx_4}{df_{t_0,t_4}}\right ),
\end{equation}
where
$$
f_{t_0,t_4}:=-t_4-x_1^5-x_2^5-x_3^5-x_4^5+5t_0x_1x_2x_3x_4.
$$
\section{Gauss-Manin connection, I}
\label{gmI}
We would like to calculate the Gauss-Manin connection 
$$
 \nabla:H_{\dR}^{3}(W/S)\to \Omega_S^1\otimes_{\O_S}H_{\dR}^{3}(W/S).
$$
of the  two parameter proper  family of 
varieties $W_{t_0,t_4},\ (t_0,t_4)\in S$.
By abuse of notation we use $\frac{\partial }{\partial t_i},\ i=0,4$ 
instead of $\nabla_{\frac{\partial}{\partial t_i}}$. We calculate $\nabla$ with respect to the 
basis
$$
\omega_i= {\frac{\partial^{i-1}}{\partial t_0^{i-1}}}(\omega_1),\ i=1,2,3,4
$$
of global sections of $H^3_\dR(W/S)$. 
For this purpose we return back to the one parameter case. We set $t_0=1$ and $t_{4}=z$ and calculate the Picard-Fuchs equation of $\eta$ with respect to the parameter $z$:
$$
{\frac{\partial^{4}\eta}{\partial z^{4}}}=\sum_{i=1}^{4} a_i(z){\frac{\partial^{i-1}\eta}{\partial z^{i-1}}} \ \ \ \ \text{    modulo relatively exact forms.}
$$
This is in fact the linear differential equation 
\begin{equation}
\label{18fev2009}
I''''=\frac{-24}{625z^4-625z^3}I+
\frac{-24z+5}{5z^4-5z^3}I'+
\frac{-72z+35}{5z^3-5z^2}I''+
\frac{-8z+6}{z^2-z}I'''
\end{equation}
which is calculated in \cite{can91}, see also \cite{ho06-1} for some algorithms which calculate such differential equations. 
It is satisfied by the periods 
$I(z)=\int_{\delta_z}\eta,\ \delta\in H_3(W_{1,z},\Q)$ of the differential form $\eta$ on the the 
family $W_{1,z}$. 
In the basis ${\frac{\partial^{i}\eta}{\partial z^{i}}},\ \ i=0,1,2,3$ the Gauss-Manin 
connection matrix has the form
\begin{equation}
\label{06jul2010}
A(z)dz:=
\begin{pmatrix}
0&1&0&0\\
0&0&1&0\\
0&0&0&1\\
a_1(z)&a_{2}(z)&a_{3}(z)&a_{4}(z)\\
\end{pmatrix}dz.
\end{equation}
Now, consider the identity map 
$$
g:W_{{(t_0,t_{4})}}\to W_{1,z},\  
$$
which satisfies $g^*\eta=t_0\omega_1$. Under this map
$$
\frac{\partial}{\partial z}=\frac{-1}{5}\frac{t_0^{6}}{t_{4}}\frac{\partial}{\partial t_0}\left (=t_0^{5}\frac{\partial}{\partial t_{4}}\right ).
$$
From these two equalities we obtain a matrix $S=S(t_0,t_{4})$ such that
$$
[\eta, {\frac{\partial\eta}{\partial z}}, {\frac{\partial^2\eta}{\partial z^2}}, 
{\frac{\partial^{3}\eta_1}{\partial z^{3}}}]^\tr=
S^{-1}[\omega_1,\omega_2,\omega_3,\omega_{4}]^\tr, 
$$
where $\tr$ denotes the transpose of matrices, and the Gauss-Manin connection in the basis $\omega_i,\ i=1,2,3,4$ is:
$$
\left (dS+S\cdot A(\frac{t_{4}}{t_0^{5}})\cdot d(\frac{t_{4}}{t_0^{5}})\right )\cdot S^{-1}
$$
which is the following matrix after doing explicit calculations:
{\tiny
\begin{equation}
 \label{25aug2010}
\begin{pmatrix}
-\frac{1}{5t_4}dt_4
& dt_0+\frac{-t_0}{5t_4}dt_4
&0
&0
\\0
&\frac{-2}{5t_4}dt_4
&dt_0+\frac{-t_0}{5t_4}dt_4
&0
\\0
&0
&\frac{-3}{5t_4}dt_4
&dt_0+\frac{-t_0}{5t_4}dt_4
\\\frac{-t_0}{t_0^5-t_4}dt_0+\frac{t_0^2}{5t_0^5t_4-5t_4^2}dt_4
&\frac{-15t_0^2}{t_0^5-t_4}dt_0+\frac{3t_0^3}{t_0^5t_4-t_4^2}dt_4
&\frac{-25t_0^3}{t_0^5-t_4}dt_0+\frac{5t_0^4}{t_0^5t_4-t_4^2}dt_4
&\frac{-10t_0^4}{t_0^5-t_4}dt_0+\frac{6t_0^5+4t_4}{5t_0^5t_4-5t_4^2}dt_4
\end{pmatrix}
\end{equation}
}
From the above matrix or directly from (\ref{18fev2009}) one can check that the periods $x_{i1},\ i=1,2,3,4$ in the Introduction satisfy the Picard-Fuch equation:
\begin{equation}
\label{31aug2010}
I{''''}=\frac{-\psi}{\psi^5-1}I+\frac{-15\psi^2}{\psi^5-1}I'+ \frac{-25\psi^3}{\psi^5-1}I''+\frac{-10\psi^4}{\psi^5-1}I''',\ \ '=\frac{\partial}{\partial \psi}.
\end{equation}

\section{Intersection form and Hodge filtration}
\label{intersection}
For $\omega,\alpha\in H^3_\dR(W_{t_0,t_{4}})$ let
$$
\langle \omega,\alpha \rangle:=\frac{1}{(2\pi i)^3}\int_{W_{t_0,t_{4}}}\omega\cup \alpha.
$$
This is Poincar\'e dual to the intersection form in $H_3(W_{t_0,t_4},\Q)$. In 
$H^{3}_\dR(W_{t_0,t_4})$  we have the Hodge filtration
$$
\{0\}=F^4\subset F^3\subset F^2\subset F^1\subset F^0=H^{3}_\dR(W_{t_0,t_4}),\ \ \dim_\C(F^i)=4-i.
$$
There is a relation between the Hodge filtration and the intersection form which is given by the following collection of equalities:
$$
\langle F^i,F^j\rangle=0, \ i+j\geq 4. 
$$
The Griffiths transversality is a property combining the Gauss-Manin connection and the Hodge filtration.
It says that  the Gauss-Manin connection sends $F^i$ to $\Omega^1_{S}\otimes F^{i-1}$ for $i=1,2,3$. 
Using this we conclude 
that:
$$
\omega_i\in F^{4-i}, \ i=1,2,3,4.
$$ 
\begin{prop}
\label{19aug2010}
 The intersection form in the basis $\omega_i$ is:

$$
 [\langle \omega_i,\omega_j\rangle]=
\begin{pmatrix}
0             &       0     &      0    &       \frac{1}{625}(t_4-t_0^5)^{-1} \\
0             &       0     &    -\frac{1}{625}(t_4-t_0^5)^{-1}  & -\frac{1}{125}t_0^4(t_4-t_0^5)^{-2}  \\         
0             &       \frac{1}{625}(t_4-t_0^5)^{-1}&0&           \frac{1}{125}t_0^3(t_4-t_0^5)^{-2} \\      
-\frac{1}{625}(t_4-t_0^5)^{-1} & \frac{1}{125}t_0^4(t_4-t_0^5)^{-2} & -\frac{1}{125}t_0^3(t_4-t_0^5)^{-2} &0      
\end{pmatrix}
$$
\end{prop}
\begin{proof} 
Let $\Omega$ be the differential form $\omega_1$ with restricted parameters 
$t_0=\psi$ and $t_{4}=1$. We have 
$\langle5\psi \Omega,\frac{\partial^3(5\psi\Omega)}{\partial^3\psi}\rangle=
\frac{1}{5^2}\frac{\psi^2}{1-\psi^5}$ (see \cite{can91}, (4.6)). From this we get:
\begin{equation}
 \label{22july10}
\langle \omega_1,\omega_4\rangle=5^{-4}\frac{1}{t_4-t_0^5}.
\end{equation}
The corresponding calculations are as follows: In $(t_0,t_4)$ coordinates we have  $\psi={t_0}{t_4^{-\frac{1}{5}}}$
and $\frac{\partial}{\partial \psi}=
t_4^\frac{1}{5}
\frac{\partial}{\partial t_0}$ and
\begin{eqnarray*}
 \langle\psi\Omega,\frac{\partial^3\psi\tilde\eta}{\partial^3\psi}\rangle &=&
\langle t_0\omega_1,(t_4^\frac{1}{5}
\frac{\partial}{\partial t_0})^{(3)}(t_0\omega_1)\rangle \\
&=& \langle t_0\omega_1,t_0t_4^\frac{3}{5}\omega_4\rangle=  t_0^2 t_4^\frac{3}{5} \langle \omega_1,\omega_4\rangle.\\
\end{eqnarray*}
From another side 
$\frac{1}{5^2}\frac{\psi^2}{1-\psi^5}=\frac{1}{5^2} \frac{t_0^2 t_4^\frac{3}{5} }{t_4-t_0^5}$.

We make the derivation of the equalities $\langle \omega_1,\omega_3\rangle=0$ and (\ref{22july10}) with respect to $t_0$ and use the Picard-Fuchs equation of $\omega_1$ 
with respect to the parameter $t_0$ and with $t_4$ fixed:
$$
\frac{\partial \omega_4}{\partial t_0}=M_{41}\omega_1+M_{42}\omega_2+M_{43}\omega_3+M_{44}\omega_4
$$
Here, $M_{ij}$ is the $(i,j)$-entry of (\ref{25aug2010}) after setting $dt_4=0,\ dt_0=1$. 
We get
$$
\langle \omega_2,\omega_3\rangle=- \langle \omega_1,\omega_4\rangle,\ \langle \omega_2,\omega_4\rangle=\frac{\partial \langle \omega_1,\omega_4\rangle}{\partial t_0}-M_{44}
\langle \omega_1,\omega_4\rangle
$$
Derivating further the second equality we get:
$$
\langle \omega_3,\omega_4\rangle= \frac{\partial \langle \omega_2,\omega_4\rangle}{\partial t_0}-M_{43}\langle \omega_2,\omega_3\rangle-M_{44}
\langle \omega_2,\omega_4\rangle.
$$
\end{proof}

\section{Moduli space, II}
\label{moduli2}

Let $T$  be the moduli of pairs $(W,\omega)$, where $W$ is a Calabi-Yau variety as before
 and $\omega\in H_\dR^3(W)\backslash F^1$ and $F^1$ is the biggest non trivial piece of the
Hodge filtration of $H^3_\dR(W)$. In this section, we  construct good affine coordinates for the moduli space  $T$.

Let $G_m$ be the multiplicative group $(\C-\{0\},\ \cdot)$ and let $G_a$ be the additive group $(\C, +)$. Both these algebraic groups act on the moduli spaces $T$:
$$
(W,\omega)\bullet k=(W,k\omega),\ k\in G_m,\ (W,\omega)\in T,
$$
$$
(W,\omega)\bullet k=(W,\omega+k\omega'),\ k\in G_a,\ (W,\omega)\in T,
$$
where $\omega'$ is uniquely determined by
 $\langle \omega',\omega\rangle=1,\ \omega'\in F^{3}$. We would like to have 
affine coordinates $(t_0,t_1,t_2,t_3,t_{4})$ for  $T$ such that;
\begin{enumerate}
 \item 
We have a canonical map 
$$
\pi: T\to S, \ \ \ (W,\omega)\mapsto (W,\omega'),
$$
where $\omega'$ is determined uniquely by $\langle \omega',\omega\rangle=1, \ \omega'\in F^3$.
In terms of the coordinates $t_i$'s it is just the projection on $t_0,t_{4}$ coordinates.
\item
With respect to the action of $G_m$, $t_i$'s behave as bellow:
$$
t_i \bullet k= k^{i+1}t_i,\ i=0,1,\ldots,4.
$$
\item
With respect to the action of $G_a$, $t_i$'s behave as bellow:
$$
t_i \bullet k=t_i,\ i=0,2,3,4,\ \  k\in G_a,
$$
$$
t_1\bullet k=t_1+k,\ k\in G_a.
$$
\end{enumerate}
In order to construct $t_i$'s we take the family $W_{t_0,t_{4}}$ as before and three new variable $t_1,t_2,t_3$. One can verify easily  that
$$
\{(t_0,t_1,t_2,t_3,t_{4})\in \C^{5}\mid t_4(t_4-t_0^5)\not =0\}\cong T,\
$$
$$ 
(t_0,t_1,t_2,t_3,t_4)\mapsto (W_{t_0,t_{n+1}},\omega),
$$ 
where
\begin{equation}
\label{27nov2009}
 \omega=t_1\omega_1+t_2\omega_2+t_3\omega_3+\frac{\omega_{4}}{\langle \omega_1,\omega_{4} 
\rangle}.
\end{equation}

\section{Gauss-Manin connection, II}
\label{basis}
For the five parameter family $W_t, t:=(t_0,t_1,t_2,t_3,t_4)\in T$, we calculate the differential forms 
$\alpha_i,\ i=1,2,3,4$ in $T$ which are defined by the equality: 
\begin{eqnarray*}
 \nabla\omega &=& \sum_{i=1}^4 \alpha_i\otimes \omega_i,\\
\end{eqnarray*}
where $\omega$ is defined in (\ref{27nov2009}),
and we check that the $\Q(t)$ vector space spanned by $\alpha_i$ is exactly 
of dimension $4$ and so up to multiplication by a rational function in  
$\Q(t)$ there is a unique vector field $\rm Ra$ which satisfies
\begin{equation}
\label{11052010}
\alpha_i({\rm Ra})=0,\ \ i=1,2,3,4
\end{equation}
or equivalently $\nabla_{\rm Ra}\omega=0$. We calculate this vector field and get the following expression:
$$
{\rm Ra}=
(\frac{6}{5}t_0^5+\frac{1}{3125}t_0t_3-\frac{1}{5}t_4)\frac{\partial}{\partial t_0}+
(-125t_0^6+t_0^4t_1+125t_0t_4+\frac{1}{3125}t_1t_3) \frac{\partial}{\partial t_1}$$
$$
+(-1875t_0^7-\frac{1}{5}t_0^5t_1+2t_0^4t_2+1875t_0^2t_4+\frac{1}{5}t_1t_4+\frac{2}{3125}t_2t_3) \frac{\partial}{\partial t_2}+
$$
$$
(-3125t_0^8-\frac{1}{5}t_0^5t_2+3t_0^4t_3+3125t_0^3t_4+\frac{1}{5}t_2t_4+\frac{3}{3125}t_3^2)\frac{\partial}{\partial t_3}+
(5t_0^4t_4+\frac{1}{625}t_3t_4)\frac{\partial}{\partial t_4}.
$$
This appears in the first five lines of the ordinary differential equation (\ref{lovely}). 
The other pieces of this differential equation has to do with the fact that the choice of $\rm Ra$ is not unique. Let 
 $$
\alpha:=
\frac
{
t_0dt_{4}-5t_{4}dt_0
}{ 
(t_{4}-t_0^{5})t_{4} 
}.
$$
The vector field $\rm Ra$ turns to be unique after putting the condition
\begin{equation}
\label{11may2010}
 \alpha({\rm Ra})=1
\end{equation}
We have calculated ${\rm Ra}$ from (\ref{11may2010}) and (\ref{11052010}). The choice of $\alpha$ up to multiplication by a rational function is canonical (see bellow). 
However, choosing such a rational function does no seem to be canonical. 
\begin{prop}
\label{24may2010}
There is a unique basis $\tilde \omega_i,\ i=1,2,3,4$ 
of $H^3_{\dR}(W_t), \ t\in T$ such that 
\begin{enumerate}
\item 
It is compatible with the Hodge filtration, i.e. $\tilde\omega_i\in F^{4-i}\backslash F^{5-i}$.
\item
$\tilde\omega_{4}=\omega$ and 
$\langle  \tilde\omega_1,\tilde \omega_{4}\rangle=1$.
\item
The Gauss-Manin connection matrix $A$ of the family $W\to T$ in the mentioned basis 
is of the form
$$
A=\begin{pmatrix}
*& \alpha & 0&0\\
*&*&\alpha&0\\
*&*&*&b_4\alpha\\
*&*&*&*\\
\end{pmatrix}
$$
and 
$$
A({\rm Ra})=
\begin{pmatrix}
0& 1 & 0&0\\
0&0&1&0\\
0&b_2&b_3&b_4\\
0&0&0&0\\
\end{pmatrix}
$$
where $b_2,b_3,b_4\in \C[t]$.
\end{enumerate}
\end{prop}
Our proof of the above proposition is algorithmic and in fact we calculate $b_i$'s 
$$
b_2=
-\frac{72}{5}t_0^8-\frac{24}{3125}t_0^4t_3-\frac{3}{5}t_0^3t_4-\frac{2}{1953125}t_3^2
$$
$$
b_3= 12t_0^4+\frac{2}{625}t_3, \ \ 
b_4=-\frac{1}{5^7}(t_0^5-t_4)^2
$$
and $\tilde \omega_i$'s:
$$
\tilde\omega_1=\omega_1, \ \tilde\omega_2=(-t_0^4-\frac{1}{3125}t_3)\omega_1+(\frac{1}{5}t_0^5-\frac{1}{5}t_4)\omega_2,\ \tilde \omega_{4}=\omega
$$
$$
\tilde\omega_3:=(-\frac{14}{5}t_0^8+\frac{1}{15625}t_0^5t_2-\frac{1}{625}t_0^4t_3-\frac{1}{5}t_0^3t_4-\frac{1}{15625}t_2t_4-\frac{2}{9765625}t_3^2)\omega_1+
$$
$$
(\frac{3}{5}t_0^9+\frac{2}{15625}t_0^5t_3-\frac{3}{5}t_0^4t_4-\frac{2}{15625}t_3t_4)\omega_2+
(\frac{1}{25}t_0^{10}-\frac{2}{25}t_0^5t_4+\frac{1}{25}t_4^2)\omega_3.
$$
The polynomials $b_2$ and $b_3$ appear in the last line of the ordinary differential equation (\ref{lovely}). 
\begin{proof}
The equalities in the second item and $\tilde \omega_1\in F^{3}$ determine 
both $\tilde\omega_1=\omega_1,\tilde \omega_{4}=\omega$ uniquely. 
We first take the 3-forms $\tilde \omega_i=\omega_i,\ i=2,3$ as in the previous section and 
write the Gauss-Manin connection of the five parameter family of Calabi-Yau varieties 
$W_t,\ t\in T$ in the basis 
$\tilde \omega_i,\ i=1,2,3,4$:
$$
\nabla[\tilde \omega_i]_{4\times 1}=[\alpha_{ij}]_{4\times 4}[\tilde\omega_i]_{4\times 1}.
$$
We explain how to modify  $\tilde\omega_2$ and $\tilde\omega_3$ 
and get the basis in the announcement of the proposition. 
Let $R$ be the $\Q(t)$ vector space generated by $\alpha_{4,i},\ i=1,2,\ldots,4$.
It does not depend on the choice of the basis $\tilde \omega_{i}$ and we already mentioned 
that it is of dimension $4$. 
If we replace $\tilde \omega_2$ by
$\tilde \omega_2+a\tilde\omega_1$ then $\alpha_{11}$ is replaced by 
$\alpha_{11}-a\alpha_{12}$. Modulo $R$ the space of differential forms on $T$ is one 
dimensional and since $\alpha_{12}\not \in R$, we choose $a$ in such a way that 
$\alpha_{11}-a\alpha_{12}\in R$. We do this and so we can assume that
 $\alpha_{11}\in R$.
 The result of our calculations shows that $\alpha_{12}$ is a multiple of  
$t_0dt_{4}-5t_{4}dt_0$. We replace $\omega_2$ by
$r\tilde \omega_2$ with some $r\in \Q(t)$ and get the desired form 
for $\alpha_{12}$. We repeat the same procedure 
for $\tilde\omega_3$. In this step we replace $\tilde\omega_3$ by
$r_3\tilde \omega_3+r_2\tilde\omega_2+r_1\tilde\omega_1$ with some $r_1,r_2,r_3\in \Q(t)$. 
\end{proof}

\section{Polynomial Relations between periods}
\label{polyrel}
We take a basis $\delta_1,\delta_2,\delta_3,\delta_4\in H_3(W_{t_0,t_4},\Q)$ such that the intersection form in this basis is given  by:
\begin{equation}
\label{intersectionmatrix}
 \Psi:=[\langle \delta_i,\delta_j\rangle]=\begin{pmatrix}
 0&    0&    0&    -\frac{6}{5}\\
0&    0&    \frac{2}{5}& 0\\   
0&    -\frac{2}{5}&0&    2\\
\frac{6}{5}&0&  -2 &0  
\end{pmatrix}.
\end{equation}
It is also convenient  to use the basis $[\tilde \delta_1,\tilde \delta_2,\tilde \delta_3, \tilde \delta_4]=[\delta_1,\delta_2,\delta_3, \delta_4]\Psi^{-1}$. In this basis the intersection 
form is $[\langle \tilde \delta_i,\tilde \delta_j\rangle]=\Psi^{-\tr}$. Let $\omega_i,\ i=1,2,3,4$ be 
the basis of the de Rham cohomology $H^3_\dR(W_{t_0,t_4})$ constructed in \S \ref{gmI} and let $\tilde \delta_i^p\in H^3(W_{t_0,t_4},\Q)$ be the Poincar\'e dual of $\tilde \delta_i$, that is, 
it is defined by the property 
$\int_{\delta}\tilde \delta_i^p=\langle \delta,\tilde\delta_i\rangle$ for all $\delta\in H_3(W_{t_0,t_4},\Q)$. If we write 
$\omega_i$ in terms of $\tilde \delta_i^p$ what we get is:
$$
[\omega_1,\omega_2,\omega_3,\omega_4]=[\tilde \delta_1^p,\tilde \delta_2^p,\tilde \delta_3^p, \tilde \delta_4^p][\int_{\delta_i}\omega_j]
$$    
that is, the coefficients of the base change matrix are the periods of $\omega_i$'s over $\delta_i$'s and not $\tilde \delta_i$'s.
The matrix $[\int_{\delta_i}\omega_j]$ is called the period matrix associated to the basis $\omega_i$ of $H^3_\dR(W_{t_0,t_4})$ and the basis $\delta_i$ of $H_3(W, \Q)$. 
We have 
\begin{equation}
 \label{24aug10}
[\langle \omega_i,\omega_j\rangle]=[\int_{\delta_i}\omega_j]^{\tr} \Psi^ {-\tr}[\int_{\delta_i}\omega_j] .
\end{equation}
Taking the determinant of this equality we can calculate $\det([\int_{\delta_i}\omega_j])$ up to sign:
\begin{equation}
 \label{inchie}
\det(\per)=\frac{12}{5^{10}}\frac{1}{(t_4-t_0^5)^2}.
\end{equation}
There is another effective way to calculate this determinant without
the sign ambiguity. For simplicity, we use the restricted parameters $t_4=1$ and $t_0=\psi$ and the notation $x_{ij}:=\int_{\delta_i}\omega_j$ as in the Introduction.
Proposition \ref{19aug2010} and  the equality (\ref{24aug10}) gives us 6 non trivial relations between $x_{ij}$'s:
\begin{eqnarray*}
0 &=& -\frac{25}{6}x_{12}x_{21}+\frac{25}{6}x_{11}x_{22}+\frac{5}{2}x_{22}x_{31}-\frac{5}{2}x_{21}x_{32}-\frac{5}{6}x_{12}x_{41}+\frac{5}{6}x_{11}x_{42}  \\
0 &=& -\frac{25}{6}x_{13}x_{21}+\frac{25}{6}x_{11}x_{23}+\frac{5}{2}x_{23}x_{31}-\frac{5}{2}x_{21}x_{33}-\frac{5}{6}x_{13}x_{41}+\frac{5}{6}x_{11}x_{43}\\
0 &=& -\frac{25}{6}x_{14}x_{21}+\frac{25}{6}x_{11}x_{24}+\frac{5}{2}x_{24}x_{31}-\frac{5}{2}x_{21}x_{34}-\frac{5}{6}x_{14}x_{41}+\frac{5}{6}x_{11}x_{44}-\frac{1}{625(\psi^5-1)}\\
0 &=& -\frac{25}{6}x_{13}x_{22}+\frac{25}{6}x_{12}x_{23}+\frac{5}{2}x_{23}x_{32}-\frac{5}{2}x_{22}x_{33}-\frac{5}{6}x_{13}x_{42}+\frac{5}{6}x_{12}x_{43}+\frac{1}{625(\psi^5-1)}\\
0 &=&-\frac{25}{6}x_{14}x_{22}+\frac{25}{6}x_{12}x_{24}+\frac{5}{2}x_{24}x_{32}-\frac{5}{2}x_{22}x_{34}-\frac{5}{6}x_{14}x_{42}+\frac{5}{6}x_{12}x_{44}-\frac{\psi^4}{125(\psi^5-1)^2}\\
0 &=& -\frac{25}{6}x_{14}x_{23}+\frac{25}{6}x_{13}x_{24}+\frac{5}{2}x_{24}x_{33}-\frac{5}{2}x_{23}x_{34}-\frac{5}{6}x_{14}x_{43}+\frac{5}{6}x_{13}x_{44}+\frac{\psi^3}{125(\psi^5-1)^2}.
\end{eqnarray*}
These equalities correspond to the entries $(1,2),(1,3),(1,4),(2,3),(2,4)$ and  $(3,4)$ of (\ref{24aug10}). In the ideal of $\Q(\psi)[x_{ij}, \ i,j=1,2,3,4]$
generated by the polynomials $f_{12},f_{13},f_{14},f_{23},f_{2,4},f_{34}$ in the right hand side of the above equalities the polynomial $\det([x_{ij}])$ is reduced to 
the right hand side of (\ref{inchie}). For instance, Singular 
check this immediately (see \cite{GPS01}).  Let $y_{ij}$ be indeterminate variables, $R=\C(\psi)[y_{ij}, i,j=1,2,3,4]$ and 
$$
I:=\{f\in R\mid f(x_{ij})=0\}.
$$
\begin{prop}
\label{27aug2010}
The ideal $I$ is generated by $f_{12},f_{13},f_{14},f_{23},f_{2,4},f_{34}$.
\end{prop}
\begin{proof}
 Let $E$ be the differential field over $F=\C(\psi)$ generated by $x_{ij}$'s. Note that
the matrix $[x_{ij}]$ is the fundamental system of the linear differential equation:
$$
\frac{\partial}{\partial \psi}[x_{ij}]=[x_{ij}]B(\psi)^\tr, 
$$
where $B(\psi)$ is obtained from the matrix (\ref{25aug2010}) by putting $dt_0=1,\ dt_4=0,\ t_0=\psi,\ t_4=1$. 
The homology group $H_3(W_\psi,\Q)$ has a symplectic basis and hence 
the monodromy group of $W_\psi$ is a subgroup of ${\rm Sp}(4,\Z)$. Consequently, the differential Galois group  
$G(E/F)$  is an algebraic subgroup of ${\rm Sp}(4,\C)$ and it contains a maximal 
unipotent matrix 
which is the monodromy around $z=0$. 
By a result of Saxl and Seitz, see \cite{sax97}, we have $G(E/F)={\rm Sp}(4,\C)$. 
Therefore, $\dim G(E/F)=10$ which is the transcendental degree of the field $E$ over $F$ (see \cite{put03}).
\end{proof}

%
%
%
%
%
%
%
%
%

\section{A leaf of $\rm Ra$}
\label{12may2010}
The solutions of the the vector field $\rm Ra$ in the moduli space $T$ are the locus of parameters such that all the 
periods of $\omega$ are constant. We want to choose a solution of $\rm Ra$ and write it in an explicit form. We proceed as follows:

Let $\tilde\delta_i, \delta_i\ i=1,2,3,4$ be two basis of $H_3(W_{t},\Q)$ as in \S \ref{polyrel} and let $C_{4\times 1}=[c_1,c_2,c_3,c_4]^\tr$ be given by the 
equality 
$$
[\langle \tilde \delta_i,\tilde\delta_j\rangle] C=[1,0,0,0]^{\tr}.
$$
and so $C=[0,0,0,-\frac{6}{5}]^\tr$. 
We are interested on the loci $L$ of parameters $s\in T$ such that
\begin{equation}
\label{badbad}
\int_{\delta_i}\omega=c_i,\ i=1,2,3,4.
\end{equation}
We will write each coordinate of $s$ in terms of periods: first we note that, on this locus 
we have
$$
\int_{\delta_1}\omega_1=1 
$$
because
\begin{eqnarray*}
1 &=& \langle \omega_1,\omega\rangle=\sum_{i,j}\langle \tilde \delta_i,\tilde \delta_j\rangle \int_{\delta_i}\omega_1 \int_{\delta_j}\omega
\\ &=&
[\int_{\delta_1}\omega_1,\ldots, \int_{\delta_{4}}\omega_1][\langle \tilde \delta_i,\tilde \delta_j\rangle]C
=\int_{\delta_1}\omega_1.
 \end{eqnarray*}
By our choice $\omega_1$ does not depend on $t_1, t_2$ and $t_3$. Therefore, 
the locus of parameters $s$  in $T$ such that $\int_{\delta_1}\omega_1=1$ is given by
\begin{equation}
 \label{t_0t_4}
(s_0,s_4)=(t_0,t_4)\bullet \int_{\delta_1}=(t_0\int_{\delta_1}\omega_1,  t_4(\int_{\delta_1}\omega_1)^{5})
\end{equation}
with arbitrary $s_1, s_2,s_3$. This is because for $k=(\int_{\delta_1} {\omega_1})^{-1}$, we have $\int_{\delta_1}k\omega_1=1$ and 
under the identification $(t_0,t_{4})\mapsto (W_{t_0,t_{4}},\omega_1)$, the pair  $(t_0,t_{4})\bullet k^{-1}$ is mapped to 
$(W_{t_0,t_{4}},k\omega_1)$.
To find $s_1,s_2,s_3$ parameters we proceed as follows: we know that
 $\omega=s_1\omega_1+s_2\omega_2+s_3\omega_3+\frac{\omega_{4}}{\langle \omega_1,\omega_{4} \rangle}$. 
This together with (\ref{badbad}) and (\ref{22july10}) imply that
$$
[\int_{\delta_i}\omega_j]_{4\times 4}[s_1, s_2, s_3, 625(s_4-s_0^5)]^{\tr}=C
$$
which gives formulas for $s_1,s_2,s_3$ in terms of periods.  
Let us write all these in terms of the periods of the one parameter family $W_\psi$. Recall the notation $x_{ij}$ in the Introduction. We have
$s_0=\psi x_{11}$, $s_1=x_{11}^5$ and  
$$
\int_{\delta_i}\omega_j=x_{11}^{-j}x_{ij}.
$$ 
Note that in the above equality the cycle $\delta_i$ lives in $W_{\psi x_{11},x_{11}^5}$. 
 We restrict  $s_i$'s to $t_0=\psi,\ t_4=1$, we use the equality (\ref{inchie}) and we get:
\begin{eqnarray*}
s_k &=& 
-\frac{6}{5}\frac{(-1)^{4+k}\det [x_{11}^{-j}x_{ij}]_{i,j=1,2,3,4,\ i\not =4,\ j\not= k}}{\det [x_{11}^{-j}x_{ij}]} \\
&=& -\frac{6}{5}\frac{5^{10}}{12}(1-\psi^5)^2 (-1)^{4+k}x_{11}^k\det [x_{ij}]_{i,j=1,2,3,4,\ i\not =4,\ j\not= k} \  \ k=1,2,3,
\end{eqnarray*}
Modulo the ideal $I$ in \S\ref{polyrel} the expressions for $s_k$'s can be reduced to to the shorter expressions in the right hand side of the equalities in Theorem
\ref{main3}.  In the left hand side we have written $t_i$ instead of $s_i$. 
We also get the relation
$$
625x_{11}^5(1-\psi^5)= -\frac{6}{5}\frac{5^{10}}{12}(1-\psi^5)^2  x_{11}^4\det [x_{ij}]_{i,j=1,2,3} .
$$
The function $\psi\to s(\psi):=(s_0(\psi),s_1(\psi),\ldots,s_4(\psi))$ is tangent to the vector field $\rm Ra$ but it is not its solution. 
In order to get a solution, one has to make a change of variable in $\psi$.

\section{The parametrization}
\label{omidcreche}
Let $\tilde \omega_i,\ i=1,2,3,4$ be the basis of the de Rham cohomology of $W_{t}, \ t\in T$ constructed 
in Proposition \ref{24may2010}. 
We consider the period map:
$$
\per: T\to \Mat(4),\ t\mapsto [\int_{\delta_i}\tilde \omega_j]_{4\times 4},
$$
where  $\Mat(4)$ is the set of $4\times 4$ matrices. By our construction of $\tilde \omega_i$, its image
is of dimension $5$ and so it is an embedding in some open neighborhood $U$ of a point 
$p\in L$  in $T$ . We restrict  its inverse 
$s=(s_0,s_1,s_2,s_3,s_4)$ to  $\per(L)$,  where $L$ is defined in \S\ref{12may2010}. Note that a point in $\per(L)$ is of the form:
$$
P=
\begin{pmatrix}
 1& p_{12} & p_{13} &  0\\
 \tau & p_{22} & p_{23} & 0\\
 p_{31}& p_{32} & p_{33} & 0\\
 p_{41}& p_{42} & p_{43} & -\frac{6}{5}\\
\end{pmatrix}.
$$
We consider $s_0,s_1,s_2,s_3,s_{4}$ and all the quantities $p_{ij}$ as functions of $\tau$ 
and set $\dot a=\frac{\partial a}{\partial \tau}$. This is our derivation in (\ref{lovely}). Note that $\tau$ as a function in $\psi$ is 
given by:
$$
\tau=\frac{\int_{\delta_2}\Omega}{\int_{\delta_1}\Omega}.
$$   
We have  $\dot s(\tau)=x(\tau)\cdot {\rm Ra}(s(\tau))$ for some holomorphic function $x$ 
in $U\cap L$, because $\rm Ra$ 
is tangent to the locus $L$ and $s$ is a local parametrization of $L$. 
Let $A$ be the Gauss-Manin connection matrix of the family
$W_t, \ t\in T$ in the basis $\tilde \omega_i,\ i=1,2,3,4$. We have 
$d(\per)=\per\cdot  A^{\tr}$, from which it follows
$$
\begin{pmatrix}
 0& \dot p_{12} & \dot p_{13} & 0\\
 1& \dot p_{22} & \dot p_{23} & 0\\
 \dot p_{31}& \dot p_{32} & \dot p_{33}& 0\\
 \dot p_{41}& \dot p_{42} & \dot p_{43}& 0\\
\end{pmatrix}=
\begin{pmatrix}
 1& p_{12} & p_{13} &  0\\
 \tau& p_{22} & p_{23} & 0\\
 p_{31}& p_{32} & p_{33} & 0\\
 p_{41}& p_{42} & p_{43} & -\frac{6}{5}\\
\end{pmatrix}
\begin{pmatrix}
0& 0 & 0  &0\\
x& 0 & x\cdot b_2(s)   &0\\
0&x& x\cdot b_{3}(s)   &0\\
0&0& x\cdot b_4(s)&0\\
\end{pmatrix}.
$$
Here we have used the particular form of $A$ in Proposition \ref{24may2010}. 
The equalities corresponding to the entries 
$(1,i),\ i\geq 2$ together with the fact that $x\not =0, b_{4}(s)\not =0$ imply that $p_{12}=p_{13}=0$.
The equality for the entry $(2,1)$ implies that $x=\frac{1}{p_{22}}$. Using these,
we have
\begin{equation}
\label{24maio2010}
\begin{pmatrix}
 0& 0 & 0 & 0\\
 1& \dot p_{22} & \dot p_{23} & 0\\
 \dot p_{31}& \dot p_{32} & \dot p_{33} & 0\\
  \dot p_{41}& \dot p_{42} & \dot p_{43} & 0
\end{pmatrix}=
\begin{pmatrix}
 1& 0 & 0 & 0\\
 \tau & p_{22} & p_{23} & 0\\
 p_{31}& p_{32} &  p_{33} & 0\\
  p_{41}&  p_{42} &  p_{43} & -\frac{6}{5}
\end{pmatrix}
\begin{pmatrix}
0& 0 & 0& 0\\
\frac{1}{p_{22}}&0&\frac{b_{2}(s)}{p_{22}}&0\\
0&\frac{1}{p_{22}}& \frac{b_{3}(s)}{p_{22}}&0\\
0&0&
 \frac{b_{4}(s)}{p_{22}}&0
\end{pmatrix}.
\end{equation}

\section{Periods}
\label{5aug2010}
Four linearly independent solutions of (\ref{18fev2009}) are given by 
$\psi_0,\psi_1,\psi_2,\psi_3$, where
\begin{equation}
\label{zafeman}
\sum_{i=0}^3\psi_i(\tilde z)\epsilon^i+O(\epsilon^4)=\sum_{n=0}^\infty \frac{(1+5\epsilon)(2+5\epsilon)\cdots(5n+5\epsilon)}{((1+\epsilon)(2+\epsilon)\cdots (n+\epsilon))^5}\tilde z^{n+\epsilon}
,\ \ \tilde z=\frac{z}{5^5}
\end{equation}
see for instance \cite{kon95}. In fact, there are  four topological cycles 
$\delta_1,\delta_2,\delta_3,\delta_4\in 
H_3(W_{z},\Q)$ such that 
$$
\int_{\delta_{i}}\eta=\frac{(2\pi i)^{4-i}}{5^4}(i-1)!\psi_{i-1}.
$$
Performing the monodromy of (\ref{zafeman}) around $z=0$, we get the same expression 
multiplied with $e^{2\pi i\epsilon}$. 
Therefore, the monodromy $\tilde \psi_i$ of $\psi_i$ is given according to the equalities:
$$
\tilde \psi_0=\psi_0,\ \tilde\psi_1=(2\pi i)\psi_0+\psi_1,\ 
\tilde\psi_2=\frac{(2\pi i)^2}{2!}\psi_0+(2\pi i)\psi_1+\psi_2,
$$
$$
\tilde\psi_3=\frac{(2\pi i)^3}{3!}\psi_0+\frac{(2\pi i)^2}{2!}\psi_1+ (2\pi i)\psi_2+\psi_3.
$$
This implies that the topological monodromy, which acts on $H_3(W_{1,z},\Q)$, in the basis 
$\delta_i,\ i=1,2,3,4$ is given by 
\begin{equation}
\label{policiafederal2010}
M=\begin{pmatrix}
1&0&0&0\\
1& 1&0&0\\
 1&2&1&0\\
1&3&3&1
\end{pmatrix}.
\end{equation}
Further, the intersection form in this basis is $\Psi$ in (\ref{intersectionmatrix}), and the monodromy around the other singularity is
$$
\begin{pmatrix}
1&-\frac{25}{6}&0&-\frac{5}{6}\\
0&1&    0&0\\
0&0&    1&0\\  
0&0&    0&1  
\end{pmatrix}.
$$
see for instance \cite{vanvan},  page 5. In fact in \cite{vanvan} the authors have considered 
the basis $C[\delta_1,\delta_2,\delta_3,\delta_4]^\tr$, where
$$
C=\begin{pmatrix}
   0&  \frac{25}{6}&0&  \frac{5}{6}\\
\frac{25}{6}&0&   \frac{5}{2}&0\\  
0&   5&   0&  0\\  
5&   0&   0&  0  
  \end{pmatrix}
$$
Note that in the mentioned reference when the authors say that with respect to a basis $\delta_1,\delta_2,\delta_3,\delta_4$ of a vector space,
a linear map is given by the matrix $T$ then the action of the linear map on $\delta_i$ is the $i$-th coordinate
of $[\delta_1,\delta_2,\delta_3,\delta_4]T$ and not $T[\delta_1,\delta_2,\delta_3,\delta_4]^\tr$.
Define
$$
Z=\begin{pmatrix}
   1&0&0&0\\
\tau&1&0&0\\
\tau^2&2\tau&2&0\\
\tau^3&3\tau^2&6\tau&6
\end{pmatrix}.
$$
Note that
$$
D=Z^{-1}\dot Z= 
\begin{pmatrix}
 0&0&0&0\\
1&0&0&0\\
0&1&0&0\\
0&0&1&0 
\end{pmatrix}
$$
and under the monodromy $M$, $\tau$ goes to $\tau+1$ and $Z$ goes to $MZ$. Therefore
$$
Q=Z^{-1}P=
\begin{pmatrix}
1&0&0&0\\
0&p_{22}&p_{23}&0\\
\frac{1}{2}(p_{31}-\tau^2)& \frac{1}{2}p_{32}-\tau p_{22} & \frac{1}{2}p_{33}-\tau p_{23}& 0\\
\frac{1}{3}\tau ^3-\frac{1}{2}\tau p_{31}+\frac{1}{6}p_{41}&
\frac{1}{2}\tau ^2p_{22}-\frac{1}{2}\tau p_{32}+\frac{1}{6}p_{42}&
\frac{1}{2}\tau ^2p_{23}-\frac{1}{2}\tau p_{33}+\frac{1}{6}p_{43})& -\frac{1}{5}
\end{pmatrix}
$$
is invariant under the monodromy around $0$. The differential equation of $P$ is given in
(\ref{24maio2010}) which we write it in the form
$\dot P=\frac{1}{p_{22}}P\cdot A({\rm Ra})^\tr$.
From this we calculate the differential equation of $Q$;
$$
\dot Q=-Z^{-1}\dot ZZ^{-1}P+Z^{-1}\dot P=-DQ+\frac{1}{q_{22}}QA(Ra)^{\tr}=
$$
$$
\frac{1}{q_{22}}
\begin{pmatrix}
 0&0&0&0\\
0&q_{23}&
q_{22}b_{2}+q_{23}b_{3}&0\\
q_{32}& -q_{22}^2+q_{33} &
q_{32}b_{2}+q_{33}b_{3}-q_{22}q_{23} & 0 \\
-q_{22}q_{31}+q_{42}& 
-q_{22}q_{32}+q_{43}&
q_{42}b_{2}+q_{43}b_{3}-\frac{1}{5}b_{4}-q_{22}q_{33}& 
0
\end{pmatrix}.
$$
Let us use the new notation $s_5=q_{22}$ and $s_6=q_{23}$. The first five lines of our 
differential equation (\ref{lovely}) is  just $\dot s=\frac{1}{s_5}{\rm Ra}(s)$ and the next two 
lines correspond to the equalities of $(2,2)$ and $(2,3)$ entries of 
the above matrices. Note that in (\ref{lovely}) we have used the notation $t_i$ instead of $s_i$.

\section{Calculating $q$-expansions}
\label{qexpension}
All the quantities $s_i$ are  invariant under the monodromy $M$ around $z=0$. 
This implies that they are invariant under the transformation $\tau\to \tau+1$. 
Therefore, all $s_i$'s can be written in terms of 
the new variable $q=e^{2\pi i \tau}$.
In order to calculate all these $q$-expansions, it is enough to restrict to the case 
$t_0=1, t_1=t_2=t_3=0, t_4=z$. 
We want to write
$$
s_0=\int_{\delta_1}\eta,\ s_4=z(\int_{\delta_1}\eta)^5
$$
in terms of $q$. Calculating $\psi_0$ and $\psi_1$ from the formula (\ref{zafeman}) we get:
$$
\psi_0=\sum_{m=0}^{\infty}\frac{(5m)!}{(m!)^5}\tilde z^m
$$
$$
\psi_1=\ln(\tilde z) \psi_0(\tilde z)+5\tilde \psi_1(\tilde z),\ \ \tilde\psi_1:=\sum_{m=1}^{\infty}\frac{(5m)!}{(m!)^5}(\sum_{k=m+1}^{5m}\frac{1}{k})\tilde z^m
$$
and so 
$$
q=e^{2\pi i\frac{\int_{\delta_2}\eta}{\int_{\delta_1}\eta}}=\tilde z e^{5\frac{\tilde \psi_1(\tilde z)}{\psi_0(\tilde z)}}.
$$
By comparing  few 
coefficients of $\tilde z^i$ and we get 
\begin{equation}
\label{tanhayi0}
s_0=\int_{\delta_1}\eta=\frac{1}{5}(\frac{2\pi i}{5})^3\psi_0=
\frac{1}{5}(\frac{2\pi i}{5})^3(1+5!q+21000q^2+\cdots)
\end{equation}
\begin{equation}
\label{tanhayi1}
s_4=z(\int_{\delta_1}\eta)^5=5^5(\frac{1}{5}(\frac{2\pi i}{5})^3)^5  \tilde z \psi_0^5= (\frac{2\pi i}{5})^{15}(0+q-170q^2+\cdots).
\end{equation}
In the differential equation (\ref{lovely}), we consider the weights
\begin{equation}
\label{7sep2010}
\deg(t_i)=3(i+1),\ i=0,1,\ldots,4, \ \deg(t_5)=11,\  \deg(t_6)=23. 
\end{equation}
In this way in its right hand side we have homogeneous rational functions of degree $4,7,10,13,16,12,24$ which is compatible with the left hand side if we assume that the derivation increases the degree by one. We have 
$\frac{\partial}{\partial \tau}=(\frac{2\pi i}{5})5q\frac{\partial}{\partial q}$ and so $(\frac{2\pi i}{5})^{-\deg(t_i)}s_i,\ i=0,1\ldots,6$ is the solution presented in the Introduction. 
The initial values (\ref{22july2010}) in the Introduction are taken from the equalities (\ref{tanhayi0}) and (\ref{tanhayi1}). In the literature, see for instance \cite{kon95, pan97}, we find also the equalities:
$$
q_{31}=\frac{1}{2}(p_{31}-\tau ^2)= \frac{1}{2}(\frac{\int_{\delta_3}\eta}{ \int_{\delta_1}\eta }- (\frac{\int_{\delta_2}\eta}{ \int_{\delta_1}\eta })^2)=
\frac{1}{(2\pi i)^2}(\frac{\psi_2}{\psi_0}-\frac{1}{2}(\frac{\psi_1}{\psi_0})^2)=
$$
$$
\frac{1}{(2\pi i)^2}\frac{1}{5}( \sum_{n=1}^\infty (\sum_{d|n}n_d d^3)\frac{q^n}{n^2}),
$$
$$
q_{14}=\frac{1}{3}\tau ^3-\frac{1}{2}\tau p_{31}+\frac{1}{6}p_{41}=
\frac{1}{(2\pi i)^3}( \frac{1}{3}(\frac{\psi_1}{\psi_0})^3-\frac{\psi_1}{\psi_0}\frac{\psi_2}{\psi_0}+\frac{\psi_3}{\psi_0})=
\frac{2}{5}\frac{1}{(2\pi i)^3}(( \sum_{n=1}^\infty (\sum_{d|n}n_d d^3)\frac{q^n}{n^3}) ), 
$$
where $n_d$ are as explained in the Introduction.

\section{Proof of Theorem \ref{main3}}
The proof of the equalities for $t_0,t_1,t_3,t_4$ is done in \S \ref{12may2010}. In \S \ref{basis} we have calculated 
$\tilde \omega_2,\ \tilde \omega_3$ in terms of $\omega_2$ and $\omega_3$. In \S \ref{omidcreche}  and \S \ref{5aug2010} we have defined
$$
s_5=p_{22}=q_{22}=\int_{\delta_2}\tilde \omega_2,\ \  s_6=p_{23}=q_{23}=\int_{\delta_2}\tilde \omega_3.
$$
Using $\int_{\delta_i}\omega_j=x_{11}^{-j}x_{ij}$ we get the expressions for $s_5,s_6$ in Theorem \ref{main3}.
Note that for simplicity in Theorem \ref{main3} we have again used the notation $t_i$ instead of $s_i a^{-\deg(t_i)}$, where $a=\frac{2\pi i}{5}$ and $\deg(t_i)$ is defined 
in (\ref{7sep2010}).

\section{Proof of Theorem \ref{29.03.10}}
The Yukawa coupling $k_{\tau\tau\tau}$ is a quantity attached to the family of Calabi-Yau 
varieties $W_{1,z}$. It can be written in terms of periods:
$$
k_{\tau\tau\tau}=
\frac{-5^{-4}a^6}{(z\frac{\partial\tau}{\partial z})^3(z-1)(\int_{\delta_1}\eta)^2},
$$
where $\tau=\frac{\int_{\delta_2}\eta}{\int_{\delta_1}\eta}$ and $a=\frac{2\pi i}{5}$, 
see for instance \cite{mo92} page 258. In \cite{can91} the authors have calculated the $q$-expansion
of the Yukawa coupling and they have reached to spectacular predictions presented in Introduction.
Let us calculate the Yukawa coupling in terms of our auxiliary quantities $s_i$. We use the notation $t_i=s_i a^{-\deg(t_i)}$.
\begin{eqnarray*}
 k_{\tau\tau\tau} &=& \frac{-5^{-4}a^6 }{(\frac{t_4}{t_0^5})^3\left (\frac{\partial\left (\frac{t_4}{t_0^5}\right )}{\partial \tau}\right )^{-3}(\frac{t_4}{t_0^5}-1)(a^3t_0)^2}
=
\frac{-5^{-4}\left (\dot{\overbrace{\frac{t_4}{t_0^5}}}\right)^3   }{
(\frac{t_4}{t_0^5})^3(\frac{t_4}{t_0^5}-1)t_0^2
}
=
\frac{ -5^{-4}(t_0\dot t_4-5\dot t_0 t_4)^3 t_0^{12}}
{t_4^3(t_4-t_0)} \\
&=&
\frac{-5^{-4} (t_0 (5t_0^4t_4+\frac{1}{625}t_3t_4)-5(\frac{6}{5}t_0^5+\frac{1}{3125}t_0t_3-\frac{1}{5}t_4) t_4)^3 t_0^{12}}
{t_5^3t_4^3(t_4-t_0)} \\
&=&
\frac{-5^{-4}(t_4-t_0^5)^2}{t_5^3}
\end{eqnarray*}
Theorem \ref{29.03.10} is proved. 


\section{Proof of Theorem \ref{29march10}}
First, we note that if there is a polynomial relation with coefficients in $\C$
between $t_i, i=0,1,\ldots,6$ (as power series  in $q=e^{2\pi i \tau}$ and hence as functions in $\tau$) then the 
same 
is true if we change the variable $\tau$ by some function in another variable.
In particular, we put $\tau=\frac{x_{21}}{x_{11}}$ and obtain $t_i$'s in terms of periods. Now, it is enough to prove that the period expressions in Theorem \ref{main3} 
 are algebraically independent over $\C$. Using Proposition \ref{27aug2010}, it is enough to prove that the variety induced by the ideal 
$\tilde I=\langle t_i-k_i,\ i=0,1,\ldots,6\rangle+I\subset k[y_{ij},i,j=1,2,3,4]$  is of dimension 
$16-6-7=3$. Here $k_i$'s are arbitrary parameters, $I$ is the ideal in \S\ref{polyrel}, $\k=\C(k_i,i=0,1,\ldots,6)$ and in the expressions of $t_i$ we have written $y_{ij}$ instead of $x_{ij}$.
This can be done by any software in commutative algebra (see for instance \cite{GPS01}).




\section{Where is Calabi-Yau monster?}
The parameter $j=z^{-1}=\frac{t_0^5}{t_4}$ classifies the Calabi-Yau varieties of type (\ref{shahva}), that is, each such Calabi-Yau 
variety is represented exactly by one value of $j$ and two such Calabi-Yau varieties are isomorphic if and only if the corresponding $j$ values are equal. 
This is similar to the case of elliptic curves which are classified by the $j$-function (see \S \ref{modularforms}). 
We have calculated also the $q$-expansions of $j$:
$$
3125\cdot j=\frac{1}{q}+770
+421375q
+274007500q^2
+236982309375q^3
+251719793608904q^4
$$
$$
+304471121626588125q^5
+401431674714748714500q^6
+562487442070502650877500q^7
$$
$$
+824572505123979141773850000q^8
+1013472859153384775272872409691q^9+O(q^{10})
$$
The coefficient $3125$ is chosen in such a way that all the coefficients of $q^i, i\leq 9$ in 
$3125\cdot j$ are integer and all together are relatively prime. Note that the moduli parameter
$j$ in our case has two cusps $\infty$ and $1$, that is, for these values of $j$ we have
singular fibers. Our $q$-expansion is written around the cusp $\infty$.

All the beautiful history behind the interpretation of the coefficients of the $j$-function of elliptic curves, monster group, 
monstrous moonshine conjecture and Borcherds proof, may indicate us another fascinating mathematics behind the $q$-expansion  of 
the $j$-function of the varieties (\ref{shahva}).

\section{A conjecture}
\label{conjecture}
 We have calculated the first eleven coefficients of 
\begin{equation}
\label{mardesic}
\frac{1}{24}t_0,\frac{-1}{750}t_1,\frac{-1}{50}t_2,\frac{-1}{5}t_3,-t_4,25t_5,15625t_6
\end{equation}
in the differential equation (\ref{lovely}).

$\frac{1}{24}t_0=\frac{ 1}{120}+q+175q^2+
   117625q^3+
   111784375q^4+
   126958105626^5+
   160715581780591q^6+
   218874699262438350q^7+
   314179164066791400375q^8+
   469234842365062637809375q^9+
   $\\$
   722875994952367766020759550q^{10}+O(q^{11})$
\vspace{.2in}

$\frac{-1}{750}t_1=\frac{1}{30}+
   3q+
   930q^2+
   566375q^3+
   526770000q^4+
   592132503858q^5+
   745012928951258q^6+
   1010500474677945510q^7+
   1446287695614437271000q^8+
   2155340222852696651995625q^{9}+$\\
   $3314709711759484241245738380q^{10}+O(q^{11})$
\vspace{.2in}

 $\frac{-1}{50}t_2=\frac{7}{10}+
   107q+
   50390q^2+
   29007975q^3+
   26014527500q^4+
   28743493632402q^5+$\\$
   35790559257796542q^6+
   48205845153859479030q^7+
   68647453506412345755300q^8+$\\$
   101912303698877609329100625q^9+
   156263153250677320910779548340q^{10}+O(q^{11})$
\vspace{.2in}

$\frac{-1}{5}t_3=\frac{6}{5}+
   71q+
   188330q^2+
   100324275q^3+
   86097977000q^4+
   93009679497426q^5+$\\$
   114266677893238146q^6+
   152527823430305901510q^7+
   215812408812642816943200q^8+$\\$
   318839967257572460805706125q^9+
   487033977592346076373921829980q^{10}+O(q^{11})$
   
\vspace{.2in}

$-t_4=
   0
   -1q^1+
   170q^2+
   41475q^3+
   32183000q^4+
   32678171250q^5+
   38612049889554q^6+
   50189141795178390q^7+
   69660564113425804800q^8+
   101431587084669781525125q^{9}$\\$
   153189681044166218779637500q^{10}+O(q^{11})$
   \vspace{.2in}

   $25t_5=\frac{-1}{125}+
   15q+
   938q^2+
   587805q^3+
   525369650q^4+
   577718296190q^5+
   716515428667010q^6+
   962043316960737646q^7+
   1366589803139580122090q^8+
   2024744003173189934886225q^9+$\\$
   3099476777084481347731347688q^{10}+O(q^{11})$
   \vspace{.2in}

$15625t_6=0
   -15q+
   26249q^2+
   3512835q^3+
   2527019900q^4+
   2381349669050q^5+$\\$
   2699403828169815q^6+
   3414337117855753978q^7+
   4647615139046603293280q^8+$\\$
   6668975996587015549602975q^9+
   9957519516309695103093241870q^{10}+O(q^{11})
   $
\\   
We have also calculated  the Yukawa coupling 
$\frac{-(t_4-t_0^5)^2}{625t_5^3}$. The numbers $n_s$ in the Introduction are given by:
\\
$ 5,2875,609250,317206375,242467530000,229305888887625,248249742118022000,$\\$295091050570845659250,375632160937476603550000, 503840510416985243645106250,$\\$ 704288164978454686113488249750
$\\
Based on these calculations we may conjecture:
\begin{conj}
\label{moussu}
All $q$-expansions of 
$$
\frac{1}{24}t_0-\frac{ 1}{120},\ 
\frac{-1}{750}t_1-\frac{1}{30},\ 
\frac{-1}{50}t_2-\frac{7}{10},\ 
\frac{-1}{5}t_3-\frac{6}{5},\ 
-t_4,\ 
25t_5+\frac{1}{125},\ 
15625t_6
$$
have positive integer coefficients.
\end{conj}
We have verified the conjecture  for the coefficients of $q^i, i\leq 50$ (see the author's web page).
The rational numbers  which appear  in (\ref{mardesic}) are chosen in such a way that the coefficients 
$t_{i,n}, n=1,2,\ldots,10$ become positive integers and for each fixed $i$ they are relatively prime. 
Writing the series $t_i$ as Lambert series $a_0+\sum_{d=1}^\infty a_d\frac{q^d}{1-q^d}$ does not help  for understanding the structure of $t_{i,n}$. It is not possible to factor out some potential of $d$ from $a_d$'s for each $t_i$.  One should probably take out a polynomial in $q$ from $t_i$ and then try to understand the nature of the sequences.   

I gave  the conjecture (\ref{moussu}) in the case of Ramanujan differential equation (\ref{raman}) to my students in a number theory course 
(the first initial values $t_{1,0}=1, \ t_{1,1}=-24$ are enough to determine all coefficients uniquely). They were not aware about Eisenstein series. They calculated some first coefficients  and then using the on-line encyclopedia of integer sequences they guessed the general formula (\ref{eisenstein}).
The mentioned encyclopedia does not recognize the integer sequences  of 
$t_0,t_1,\ldots,t_6$. 
This support the fact that the general formula for $t_i$'s or any interpretation of them  is not yet known.

\section{Moduli space, III}
\label{moduli3}
In this section we introduce moduli interpretation for $t_5$ and $t_6$.
Let $\tilde {\rm Ra}$ be the vector field in $\C^7$ corresponding to (\ref{lovely}) and let $\tilde \omega_i,\ i=1,2,3,4$ be 
the differential forms calculated in Proposition \ref{24may2010}. Consider $t_i,\ i=0,1,2,\ldots,6$ as unknown parameters. We define a new basis $\hat\omega_i, i=1,2,3,4$ of $H^3_\dR(W_{t_0,t_4})$:
$$
\hat\omega_1=\tilde\omega_1,\ \hat\omega_2=\frac{1}{t_5}\tilde\omega_2, \ \hat\omega_3=\frac{5^7}{(t_4-t_0^5)^2}(-t_6\tilde \omega_2+t_5\tilde\omega_3),\ \hat\omega_4=\tilde \omega_4.
$$
The intersection form  in the basis $\hat\omega_i,\ i=1,2,3,4$ is a constant matrix and in fact it is:
\begin{equation}
\label{31aug10}
\begin{pmatrix}
 0&0&0&1\\
0&0&1&0\\
0&-1&0&0\\
-1&0&0&0
\end{pmatrix}
\end{equation}
The Gauss-Manin connection composed with $\tilde {\rm Ra}$ has also the form: 
$$
\nabla_{\tilde {\rm Ra}}=
\begin{pmatrix}
 0&1&0&0\\
0&0& \frac{(t_4-t_0^5)^2}{5^7t_5^3}&0\\
0&0&0&-1\\
0&0&0&0
\end{pmatrix}  
$$
It is interesting that the Yukawa coupling appears as the only non constant term in the above matrix.
Let $X$  be the moduli of pairs $(W,\{\alpha_1,\alpha_2,\alpha_3,\alpha_4\})$, where $W$ is a Calabi-Yau variety as before, 
$\alpha_i\in F^{4-i}\backslash F^{5-i}$, $F^i\subset H_\dR^3(W) $  is the $i$-th piece of the Hodge filtration, $\alpha_i$'s form a basis of $H_\dR^3(W)$ and the intersection form in 
$\alpha_i$'s is given by the matrix (\ref{31aug10}). We have the isomorphism
$$
\{t\in \C^7\mid t_5t_4(t_4-t_0^5)\not =0\}\cong X
$$ 
$$
t\mapsto (W_{t_0,t_{4}}, \{\hat\omega_1,\hat \omega_2,\hat \omega_3,\hat \omega_3\} )
$$
which gives the full moduli interpretation of all $t_i$'s.  


\def\cprime{$'$} \def\cprime{$'$} \def\cprime{$'$}

\end{document}